\documentclass[twoside,12pt]{article}
\usepackage{amsmath,amsthm,amssymb,amscd}
\usepackage[matrix,arrow,curve]{xy}
\usepackage[pdftex]{graphicx}

\title{$\zeta(n)$ via hyperbolic functions.}

\author{Joseph T. D'Avanzo and Nikolai A. Krylov}

\date {}

\newtheorem*{theorem}{Theorem}
\newtheorem*{corollary}{Corollary}
\newtheorem{lemma}{Lemma}

\newtheorem{dfn}{Definition}

\def\natu       {\mathbb N}

\def\real       {\mathbb R}

\def\ArcCosh    {\rm arcosh}
\def\Cosh       {\rm cosh}
\def\ArcSinh    {\rm arcsinh}
\def\Sinh       {\rm sinh}
\def\Tanh       {\rm tanh}
\def\Coth       {\rm coth}
\def\ArcTanh    {\rm arctanh}
\def\ArcSin     {\rm arcsin}

\begin{document}

\maketitle

\parskip=5mm

\begin{abstract}
{We present here an approach to a computation of $\zeta(2)$ by changing variables in the double integral
using hyperbolic trig functions. We also apply this approach to present $\zeta(n)$, when $n>2$, as a definite
improper integral of single variable.}
\end{abstract}


\section{Introduction}

The Riemann zeta function is defined as the series
$$
\zeta(n) = \frac{1}{1^n} + \frac{1}{2^n} +\frac{1}{3^n} + \ldots + \frac{1}{k^n} + \ldots
$$
for any integer $n\geq 2$.  Three centuries ago Euler found that $\zeta(2) = \pi^2/6$, which is an irrational number.
Exact value of $\zeta(3)$ is still unknown though it was proved by Ap\'ery in 1979 that $\zeta(3)$ was also irrational
(see \cite{Porten}).  Values of $\zeta(n)$, when $n$ is even, are known and can be written in terms of Bernoulli numbers.
We refer the interested reader to chapter 19 of \cite{BOOK} for a ``perfect" proof of the formula
$$
\zeta(2k) = \sum\limits_{n=1}^{\infty} \frac{1}{n^{2k}} = \frac{(-1)^{k-1} 2^{2k-1} B_{2k}}{(2k)!} \cdot \pi^{2k}~~~~(k\in\natu).
$$

\noindent Notice that $\zeta(n)$ can be written as the following
multi-variable integral
$$
\zeta(n) = \int_0^1\cdots  \int_0^1
\frac{1}{1-x_1x_2\cdots x_n}~dx_1dx_2\ldots dx_n.
$$
Indeed, each integral is improper at both ends and since the
geometric series $\sum\limits_{q\geq 0} x^{q}$ converges uniformly
on the interval $|x|\leq R,~\forall R\in(0,1)$ we can write
$$
\frac{1}{1-x_1 x_2\cdots x_n} = \sum\limits_{q=0}^{\infty}
(x_1x_2\cdots x_n)^q
$$
then interchange summation with integration, and then integrate $(x_1x_2\cdots x_n)^q$ for each $q$. Using the
identities
$$
\frac{1}{1-xy}+\frac{1}{1+xy}=\frac{2}{1-x^2y^2} ~~~ \text{and}~~~
\frac{1}{1-xy}-\frac{1}{1+xy}=\frac{2xy}{1-x^2y^2}
$$
and a simple change of variables one can easily see that
$$
\int_0^1{\int_0^1{\frac{1}{1-xy}~dx}dy} = \frac{4}{3}
\int_0^1\int_0^1\frac{1}{1-x^2y^2}~dxdy
$$
by further generalizing this idea one comes along the following,
$$
\zeta(n)=\frac{2^n}{2^n-1}\int_0^1\ldots\int_0^1\frac{1}{1-\prod\limits_{i=1}^n
x^2_i}~dx_1...dx_n.
$$

\noindent Notice that $(1,1)$ is the only point in the square
$[0,1]\times[0,1]$, which makes the integrand $1/(1-x^2y^2)$
singular. If we take another point on the graph of $1=x^2y^2$, say
$(a,1/a)$ with $a\in(0,\infty)$, then it follows easily (see lemma 1
below) that
$$
\int_0^{1/a}\int_0^a\frac{1}{1-x^2y^2}~dxdy =
\int_0^1\int_0^1\frac{1}{1-x^2y^2}~dxdy.
$$

\noindent This result motivates the following definition
\begin{dfn} For any point $(a_1,a_2,\ldots, a_{n-1})\in\real^{n-1}$ such that
$a_i\in(0,+\infty),~\forall i\in\{1,\ldots,n-1\}$ we define
$$
I_n(a_1,...,a_{n-1})=\int_0^{\frac{1}{a_1\cdots
a_{n-1}}}\ldots\int_0^{a_2}\int_0^{a_1}\frac{1}
{1-\prod\limits_{i=1}^n{x_i^2}}~dx_1dx_2\ldots dx_n.
$$
\end{dfn}

\begin{lemma}
For any $a_i\in(0,+\infty)$, we have $I_n(a_1,...,a_{n-1}) =
I_n(1,1,\ldots,1)$.
\end{lemma}
\begin{proof}
Simply observe that by using the change of variables $x_i=a_iu_i$
for all $i\in [1,\ldots,n]$, where $a_n=1/(a_1a_2\cdots a_{n-1})$,
the Jacobian equals 1, and the integrand is unchanged.
\end{proof}

In this article we investigate $\zeta(n)$ following Beukers, Calabi
and Kolk (see \cite{Calabi}), who used the change of variables
$$
x=\frac{\sin(u)}{\cos(v)}~~~\text{and}~~~y=\frac{\sin(v)}{\cos(u)}~~\text{to
evaluate} ~\int_0^1\int_0^1 \frac{1}{1-x^2y^2}~dxdy.
$$
Such proof of the identity $\zeta(2)=\pi^2/6$ may also be found in
chapter 6 of \cite{BOOK} and in papers of Elkies \cite{Elkies} and
Kalman \cite{Kalman}. Let us also mention here that Kalman's paper,
in addition to a few other proofs of the identity, contains some
history of the problem together with an extensive reference list.

Here we will be changing variables too, but in the integrals
$I_n(a_1,...,a_{n-1})$ and using the hyperbolic trig functions
$\sinh$ and $\cosh$ instead of $\sin$ and $\cos$. Such a change of
variables was considered independently of us by Silagadze and the
reader will find his results in \cite{Silagadze}.

\noindent {\bf Acknowledgement}: The authors would like to thank the
referee who drew our attention to Kalman's paper \cite{Kalman} and
has made a few useful suggestions that improved the exposition.

\section{Hyperbolic Change of Variables}

First observe that the change of variables
$$
x_i=\frac{\sin(u_i)}{\cos(u_{i+1})}~~~~~~\forall i\in\natu\mod(n)
$$
reduces the integrand in $I_n(1,...,1)$ to 1 only when $n$ is even. The
region of integration $\Phi_n=[(x_1,\ldots,x_n)\in\real^n: 0 <
x_1,\ldots,x_n < 1]$ becomes the one-to-one image of the
$n$-dimensional polytop (note $u_{n+1} = u_1$)
$$
\Pi_n := [(u_1,u_2,\ldots,u_n)\in\real^n: u_i > 0,
u_i+u_{i+1}<\frac{\pi}{2}, 1\leq i \le n].
$$

We suggest here a different change of variables that will
produce an integrand of 1 for all values of $n$ in
$I_n(a_1,...,a_{n-1})$. But first we define the corresponding
region.

\begin{dfn}
For any point $(a_1,a_2,\ldots, a_{n-1})\in\real^{n-1}$ such that
$a_i\in(0,+\infty),~\forall i\in\{1,\ldots,n-1\}$ we define
$$\Phi_n(a_1,a_2,\ldots,a_{n-1}):=
[(x_1,\ldots,x_n)\in\real^n~|~ 0 < x_i < a_i, \forall
i\in\{1,\ldots,n\}],
$$
where $a_n=1/(a_1\cdot a_2\cdot \ldots\cdot a_{n-1})$.
\end{dfn}

\begin{lemma} The change in variables
$$
x_i=\frac{\Sinh(u_i)}{\Cosh(u_{i+1})}  ~~~~~~~\forall i\in\natu\mod(n)
$$
reduces the integrand of $I_n(a_1,...,a_{n-1})$ to 1 for all values
of $n\geq 2$. It also gives a one-to-one differentiable map between the
region $\Phi_n(a_1,a_2,\ldots,a_{n-1})$ and the set
$\Gamma_n\subset\real^n$ described by the following $n$
inequalities:
$$
0<u_i < \ArcSinh\bigl(a_i\cdot\Cosh(u_{i+1})\bigr),~~~~~~~ \forall i\in\natu\mod(n).
$$
\end{lemma}

\begin{proof}
The inequalities for $\Gamma_n$ follow trivially from the
corresponding inequalities $0<x_i<a_i$ and the facts that
$\Cosh(x)>0$ and $\ArcSinh(x)$ is increasing everywhere. Injectivity
and smoothness of the map may be proven by writing down formulas,
which express each $u_i$ in terms of all $x_j$. For example, here
are the corresponding formulas for the set $\Gamma_3$:
$$
u_i = \ArcSinh\left( x_i\cdot\sqrt{\frac{1+x^2_{i+1} +
x^2_{i-1}x^2_{i+1}}{1-x^2_1x^2_2x^2_3}}\right),~~~i\in\natu\pmod{3}.
$$

The Jacobian is the determinant of the matrix
$$
A = \begin{pmatrix}
\frac{\cosh(u_1)}{\cosh(u_2)} & \frac{-\sinh(u_1)\sinh(u_2)}{\cosh^2(u_2)} & 0 & \ldots & 0 \\
0 & \frac{\cosh(u_2)}{\cosh(u_3)} & \frac{-\sinh(u_2)\sinh(u_3)}{\cosh^2(u_3)} & \ldots &0 \\
\vdots & \vdots & \vdots & \ddots & \vdots \\
\frac{-\sinh(u_n)\sinh(u_1)}{\cosh^2(u_1)}  & 0 & 0 & \ldots & \frac{\cosh(u_n)}{\cosh(u_1)}
\end{pmatrix}
$$

To compute this determinant we observe that the first column
expansion reduces the computation to two determinants of the upper
and lower triangular matrices. This results in the formula, where
the first term comes from the upper triangular matrix and the second
from the lower triangular matrix (recall that $u_{n+1} = u_1$) :

$$
{\rm Det}(A)=\prod\limits_{i=1}^n\frac{\cosh(u_i)}{\cosh(u_{i+1})} ~
+ ~ (-1)^{n-1}\cdot \prod\limits_{i=1}^n\frac{-\sinh(u_i)\sinh(u_{i+1})}{\cosh^2(u_{i+1})}=
1-\prod\limits_{i=1}^n\tanh^2(u_i).
$$

When using the above change in variables the denominator of the
integrand $1-\prod\limits_{i=1}^nx_i$ becomes
$1-\prod\limits_{i=1}^n\tanh^2(u_i)$, which we just proved to be the
Jacobian.
\end{proof}

\section{Computations of $\zeta(2)$}

We begin with $\zeta(2)$, which is a rational multiple of  $I_2(1)$. Lemma 1 implies that it's enough to compute
$$
I_2(a)=\int_0^{\frac{1}{a}}{\int_0^a{\frac{1}{1-x^2y^2}dx}dy} ~~~\mbox{for arbitrary}~a >0.
$$

\noindent We now preform the following change in variables
$$
x=\frac{\Sinh(u)}{\Cosh(v)}, ~ ~ y=\frac{\Sinh(v)}{\Cosh(u)}.
$$
As we proved above, our integrand reduces to 1 and all we must do is
worry about the limits. If $x= 0$ then clearly $u=0$, the same is
true for $y$ and $v$. If $x=a$ then $a\cdot\Cosh(v)=\Sinh(u)$ so
$v=\ArcCosh(\frac{\Sinh(u)}{a})$ and if $y=\frac{1}{a}$ then
$(1/a)\cdot\Cosh(u)=\Sinh(v)$ so $v=\ArcSinh(\frac{\Cosh(u)}{a})$
thus describing our region of integration (see Figure 1). We then
write the integral $I_2(a)$ as follows
$$
\int\limits_0^{\ArcSinh(a)}\ArcSinh\left(\frac{\Cosh(u)}{a}\right)du + \int\limits_{\ArcSinh(a)}^\infty\ArcSinh\left(\frac{\Cosh(u)}{a}\right)-
\ArcCosh\left(\frac{\Sinh(u)}{a}\right)du.
$$

\begin{figure*}[h]
\includegraphics[height=100mm,width=110mm]{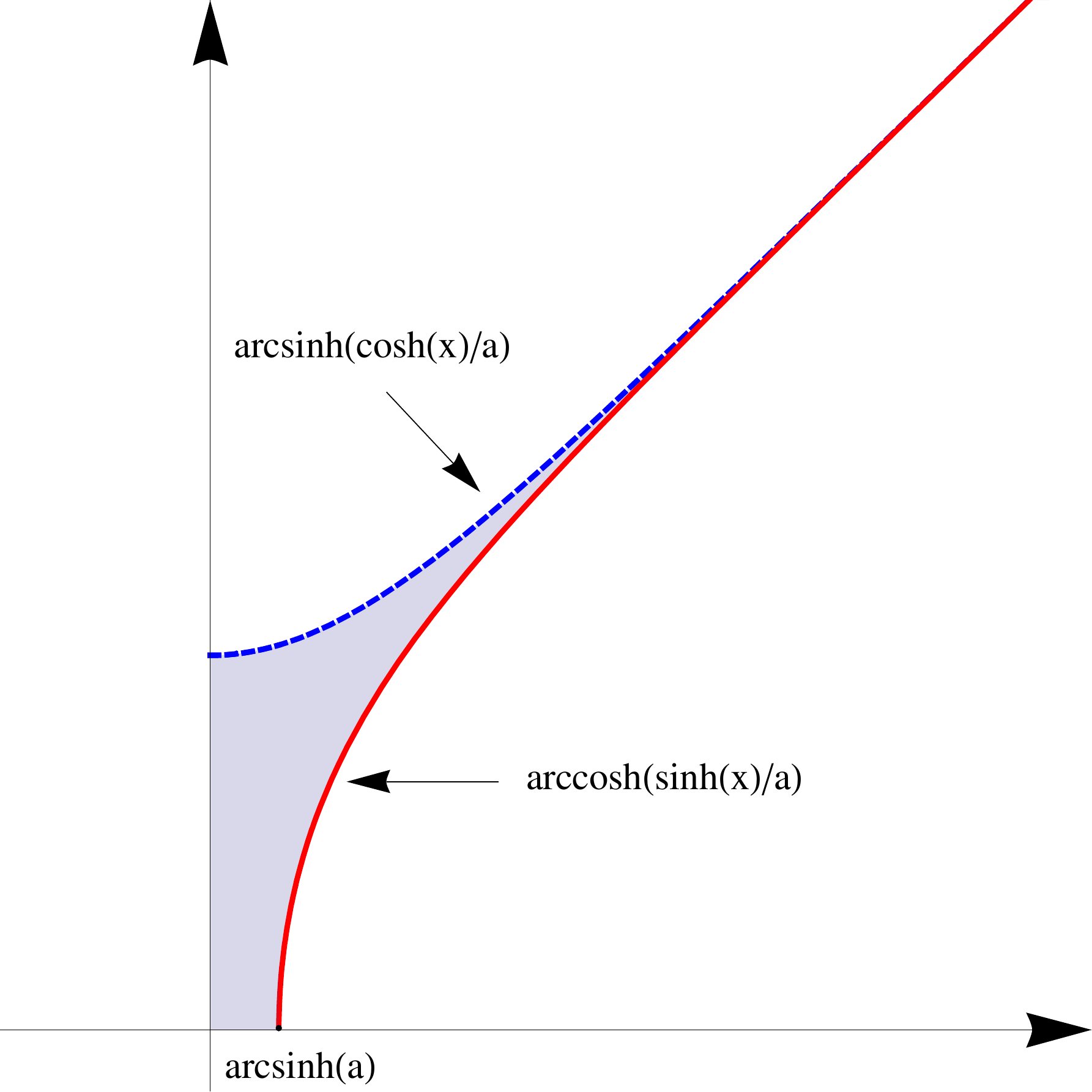}
\caption{The set $\Gamma_2\subset\real^2,~~\forall a> 0$.}
\end{figure*}

\begin{lemma}
$\lim\limits_{a\to 0} \int_0^{\ArcSinh(a)}\ArcSinh(\frac{\Cosh(u)}{a})du =0$
\end{lemma}
\begin{proof}
If we let $\Cosh(\ArcSinh(z))=Q$ then $Q=\sqrt{1+z^2}$. Therefore
$$
\ArcSinh\left(\frac{\Cosh(\ArcSinh(a))}{a}\right)=\ArcSinh\left(\sqrt{\frac{1}{a^2}+1}\right).
$$
Since $\ArcSinh(\Cosh(u)/a)$ is concave up, we can take area of the rectangle with vertices at $(0,0)$,
$({\rm arcsinh}(a),0)$, and $({\rm arcsinh}(a),{\rm arcsinh(cosh(arcsinh}(a))/a)$ as an overestimate of the integral, that is
$$
\ArcSinh(a)\cdot \ArcSinh\left(\sqrt{\frac{1}{a^2}+1}\right)\geq
\int_0^{\ArcSinh(a)}\ArcSinh\left(\frac{\Cosh(u)}{a}\right)du\geq0
$$
Then by applying L'hospital's rule one can deduce
$$
\lim\limits_{a\to0}(\ArcSinh(a)\cdot\ArcSinh\left(\sqrt{\frac{1}{a^2}+1}\right))=0.
$$
\end{proof}

Now, since $I_2(a)=I_2(1)$, $\forall a >0$, we conclude that $I_2(1)=\lim\limits_{a\to 0} I_2(a)$, and therefore we have
$$
I_2(1)=\lim_{a \to 0}
\int_{\ArcSinh(a)}^\infty{\ArcSinh\left(\frac{\Cosh(u)}{a}\right)-\ArcCosh\left(\frac{\Sinh(u)}{a}\right)du}.
$$ Since
$$
\ArcSinh\left(\frac{\Cosh(x)}{a}\right)=\ln{\left(\frac{\Cosh(x)}{a}+\sqrt{\frac{\Cosh^2(x)}{a^2}+1}\right)}
$$
 and
$$
\ArcCosh\left(\frac{\Sinh(x)}{a}\right)=\ln{\left(\frac{\Sinh(x)}{a}+\sqrt{\frac{\Sinh^2(x)}{a^2}-1}\right)}
$$
we get
$$
I_2(1)=\lim_{a\to 0}\int_{\ArcSinh(a)}^\infty{\ln{\left(\frac{\frac{\Cosh(x)}{a}+\sqrt{\frac{\Cosh^2(x)}{a^2}+1}}{\frac{\Sinh(x)}{a}+
\sqrt{\frac{\Sinh^2(x)}{a^2}-1}}\right)}dx}
$$
which, after taking the limit as $a\to0$ gives
$$
I_2(1)=\int_0^\infty\ln\left(\frac{\Cosh(x)}{\Sinh(x)}\right)dx.
$$

\noindent Using integration by parts $u=\ln{\left(\frac{\Cosh(x)}{\Sinh(x)}\right)}$ and $v=dx$ one
obtains the formula
$$
I_2(1)=\left.
x\ln{\left(\frac{\Cosh(x)}{\Sinh(x)}\right)}\right|_0^\infty+\int_0^\infty{\frac{2x}{\Sinh(2x)}dx}.
$$
By examining the limits of the first half of the formula as x goes
to 0 and $\infty$ we are left with only the integral
$$
I_2(1)=\int_0^\infty\frac{2x}{\Sinh(2x)}dx.
$$
 By applying the change in variables $u = 2x$ our formula becomes
$$
I_2(1)=\frac{1}{2}\int_0^\infty{\frac{u}{\Sinh(u)}du}.
$$

Now we use the method of differentiation under the integral sign
and consider the function
$$
F(\alpha)=\frac{1}{2}\int_0^\infty\frac{\ArcTanh(\alpha
\Tanh(x))}{\Sinh(x)}dx.
$$ One should consider the function $F$ at the points $\alpha=1$ and $\alpha=0$. $F(1)$ is clearly the integral we
are trying to find and $F(0)$ is 0. Thus by differentiating under the integral with respect to alpha, plus some algebra
we obtain
$$
F'(\alpha) = f(\alpha)=\frac{1}{2}\int_0^\infty\frac{\Cosh(x)}{1+(1-\alpha^2)\Sinh^2(x)}dx.
$$
Then by preforming the change in variables
$u=\sqrt{1-\alpha^2}\cdot \Sinh(x)$ the integral becomes
$$
f(\alpha)=\frac{1}{2\sqrt{1-\alpha^2}}\int_0^\infty\frac{1}{1+u^2}du,
$$
which is simply
$$
\left.\frac{\ArcTanh(u)}{2\sqrt{1-\alpha^2}}\right|_0^\infty =
\frac{\pi}{4\sqrt{1-\alpha^2}}.
$$

Since we took the derivative with respect to $\alpha$ we must take the integral with respect to alpha so we have
$$
\int_0^1{f(\alpha)~d\alpha}=F(1)-F(0)=F(1)-0=F(1)
$$
which, as stated above is our goal. So
$$
I_2(1) = \int_0^1 f(\alpha)~d\alpha =\frac{\pi}{4}\int_0^1\frac{1}{\sqrt{1-\alpha^2}}~d\alpha =
\left.\frac{\pi}{4}\ArcSin(\alpha)\right|_0^1 = \frac{\pi^2}{8},
$$
and hence $\zeta(2) = \frac{4}{3}\cdot\frac{\pi^2}{8} = \pi^2/6$.

\section{A formula for $\zeta(n),~n\geq 2$}

One could try to use similar approach to compute $\zeta(n),~n > 2$, however the computations become a bit long.
Instead, we present an elementary proof of the following theorem, which generalizes our formula for $\zeta(2)$ from
the previous section.

\begin{theorem} Let $n\geq 2$ be a natural number. Then
$$
\int_0^1...\int_0^1\frac{1}{1-\prod\limits_{i=1}^n{x_i^2}}~dx_1...dx_n
= \frac{1}{(n-1)!}\cdot\int_0^\infty \ln^{n-1}(\Coth(x))~dx.
$$
\end{theorem}

\noindent Let us start with the following lemma, which can be easily
proved by using induction on $k$, integration by parts and
l'Hospital's rule.

\begin{lemma}
$$
\int_0^1{\ln^{k}(z)z^{2q}dz}=\frac{(-1)^kk!}{(2q+1)^{k+1}},~~\forall
k\in \natu ~\text{and} ~ q\geq 0.
$$
\end{lemma}

\begin{proof}[Proof of the theorem.] Applying the substitution $z=\Tanh(x)$
to the integral
$$
\frac{1}{(n-1)!} \int_0^\infty \ln^{n-1}(\Coth(x))~dx
$$
gives
$$
\frac{1}{(n-1)!}\int_0^1{\frac{(-\ln(z))^{n-1}}{1-z^2}dz} =
\frac{1}{(n-1)!}\int_0^1 (-\ln(z))^{n-1}\cdot(\sum_{q\geq 0}
z^{2q})dz.
$$
Since the integral is improper at both ends and the geometric series $\sum\limits_{q\geq 0}
z^{2q}$ converges uniformly on the interval $|z|\leq R,~\forall
R\in(0,1)$, the last integral equals
$$
\frac{1}{(n-1)!}
\sum_{q\geq0}{(-1)^{n-1}\cdot \int_0^1{\ln^{n-1}(z)z^{2q}dz}} = ~ \mbox{by lemma 4} ~ = \sum_{q\geq0}{\frac{1}{(2q+1)^n}}.
$$

\noindent Using the geometric series expansion one can easily show that we also have
$$
\int_0^1\ldots\int_0^1
\frac{1}{1-\prod\limits_{i=1}^n{x_i^2}}~dx_1...dx_n =
\sum_{q\geq0}{\frac{1}{(2q+1)^n}}.
$$
\end{proof}

\begin{corollary}
For any integer $n\geq 2$,
$$
\zeta(n) = \frac{2^n}{(2^n -1)\cdot(n-1)!}\cdot\int_0^\infty \ln^{n-1}(\Coth(x))~dx.
$$
\end{corollary}

\parskip=1mm

\noindent Siena College, Department of Mathematics\\
515 Loudon Road, Loudonville NY 12211

\noindent {\small jt17dava@siena.edu {\it and} nkrylov@siena.edu}


\begin{thebibliography}{99}

\bibitem{BOOK} M. Aigner; G. M. Ziegler:~{Proofs from THE BOOK} Second Edition (2002)

\bibitem{Calabi} F. Beukers; J. Kolk; E. Calabi:~{\it Sums of generalized harmonic series and volumes,} Nieuw Arch.
Wisk. (4) 11 (1993), no. 3, 217 - 224.

\bibitem{Elkies} N. Elkies:~{\it On the sums $\sum^\infty_{k=-\infty}(4k+1)^{-n}$} Amer. Math.
Monthly 110 (2003), no. 7, 561 - 573.

\bibitem{Kalman} D. Kalman:~{\it Six Ways to Sum a Series,} The College Mathematics Journal, Vol. 24, No. 5 (Nov., 1993), 402 -
421.

\bibitem{Porten} A.van der Porten:~{\it A proof that Euler missed ...} Math. Intelligencer 1 (1979), 195 - 203.

\bibitem{Silagadze} Z. K. Silagadze:~{\it Sums of generalized
harmonic series for kids from five to fifteen}, preprint
http://arxiv.org/abs/1003.3602

\end{thebibliography}
\end{document}